\documentclass{article}
\usepackage{amssymb,latexsym}
\usepackage{amsmath}
\usepackage{graphics}
\usepackage{graphicx}
\newtheorem{theorem}{Theorem}

\newtheorem{corollary}{Corollary}

\newtheorem{definition}[theorem]{Definition}

\newenvironment{proof}[1][Proof]{\textbf{#1.} }{\ \rule{0.5em}{0.5em}}
\numberwithin{equation}{section} \numberwithin{theorem}{section}
\numberwithin{corollary}{section} \numberwithin{lemma}{section}
\numberwithin{remark}{section} \numberwithin{notation}{section}

\begin{document}
\title{Consimilarity of Split Quaternion Matrices and a Solution of the Split Quaternion Matrix Equation $X - A\widetilde XB = C$}
\author{Hidayet H\"{u}da K\"{o}sal, Mahmut Akyi\v{g}it, Murat Tosun \\
Department of Mathematics, Sakarya University, Sakarya, Turkey \\}

\maketitle

\begin{abstract}
In this paper, the consimilarity of complex matrices is generalized for the split quaternions. In this regard, coneigenvalue and coneigenvector  are defined for split quaternion matrices. Also, the existence of solution to the split quaternion matrix equation $X - A\widetilde XB = C$ is characterized and the solution of the equation in the explicit form are derived via its real representation.


\end{abstract}

\section{Introduction}

\noindent

Hamilton introduced real quaternions that can be represented as \cite{1}

\begin{equation}
\mathbb{H} = \left\{ {q = {q_0} + {q_1}i + {q_2}j + {q_3}k:\,\,{q_s} \in R}, s=0,1,2,3 \right\}
\end{equation}

\noindent where
\begin{equation}
{i^2} = {j^2} = {k^2} =  - 1,\,\,ij = - ji =k  ,\,\,jk = - kj =i  ,\,\,ki = - ik =j  .
\end{equation}

\noindent  It seems forthwith  that multiplication of the real quaternions is not commutative owing to these ruled. So, it is not easy to work the real quaternions algebra problems. Similarly, it is well known that the main obstacle in study of the real quaternions matrices, dating back 1936 \cite{2}, is the non-commutative multiplication of the real quaternions. There are many studies on matrices of the real quaternions. So, Baker discussed right eigenvalues of the real quaternion matrices with a topological approach in \cite{3}. On the other hand, Huang and So introduced on left eigenvalues of the real quaternion matrices \cite{4}. After that Huang discussed consimilarity of the real quaternion matrices and obtained the Jordan canonical form of the real quaternion matrices down below consimilarity \cite{5}. Jiang and Ling studied in \cite{6} the problem of condiagonalization of the real quaternion matrices under consimilarity and gave two algebraic methods for the condiagonalization.
Also,the existence of the solution to the real quaternion matrix equation $X - A\widetilde XB = C,$ were characterized and the solution of the equation in closed-form are derived via real representation of the real quaternion matrices, \cite{7}.

\noindent After Hamilton had discovered the real quaternions, James Cockle defined, by using real quaternions, the set of split quaternions, in 1849. The split quaternions are not commutative like real quaternions. But the set of split quaternions contains zero divisors, nilpotent and nontrivial idempotent elements \cite{8}.
The split quaternions are a recently developing topic, since the split quaternions are used to express Lorentzian relations. Also, there are many studies on geometric and physical meaning of the split quaternions \cite{8}-\cite{9}-\cite{14}. Alagoz \emph{et al.} considered split quaternion matrices. They investigated the split quaternions matrices using properties of complex matrices \cite{12}.  After that Erdogdu and Ozdemir obtained method of finding eigenvalues of the split quaternions matrices. Also, they gave an extension of Gershgorin theorem for the split quaternion matrices in \cite{10}.

\section{Consimilarity of Split Quaternions}

\noindent
Let $\mathbb{R}$, $\mathbb{C }= \mathbb{R} \oplus \mathbb{R}i$ and   ${\mathbb{H}_S} = \mathbb{R} \oplus \mathbb{R}i \oplus \mathbb{R}j \oplus \mathbb{R}k$ be the real number, complex number and split quaternion field over the $\mathbb{R}$, respectively, where

\begin{equation}
\begin{array}{l}
\,\,\,\,\,\,\,\,\,\,\,\,\,\,\,\,\,\,\,\,\,\,{i^2} =  - 1,\,\,\,\,\,\,{j^2} = {k^2} = 1\\
ij =  - ji = k,\,\,\,\,jk =  - kj =  - i,\,\,\,ki =  - ik = j.
\end{array}
\end{equation}

\noindent Let $q = {q_0} + {q_1}i + {q_2}j + {q_3}k \in {\mathbb{H}_S}.$  The conjugate and norm of split quaternion are described as, respectively,

\begin{equation}
\overline q  = {q_0} - {q_1}i - {q_2}j - {q_3}k
\end{equation}
and

\begin{equation}
\left\| q \right\| = \sqrt {\left| {q\overline q } \right|}  = \sqrt {\left| {{I_q}} \right|}
\end{equation}

\noindent where ${{I}_{q}}=q_{0}^{2}+q_{1}^{2}-q_{2}^{2}-q_{3}^{2}.$

\noindent Also, $q \in {\mathbb{H}_S}$  is said to be timelike, spacelike or null if ${{I}_{q}}>0,\,\,{{I}_{q}}<0$ or ${{I}_{q}}=0$, respectively.

\noindent The linear transformations $R ,L:{{\mathbb{H}}_{S}}\to End\left( {{\mathbb{H}}_{S}} \right),$ given by

\begin{equation}
R \left( q \right):{{\mathbb{H}}_{S}}\to {{\mathbb{H}}_{S}},\,\,\,\,R \left( q \right)\left( x \right)=xq,
\end{equation}

\noindent and

\begin{equation}
L \left( q \right):{\mathbb{H}_S} \to {\mathbb{H}_S},\,\,\,\,L \left( q \right)\left( x \right) = qx
\end{equation}

\noindent are defined the right and left representation of the algebra ${{\mathbb{H}}_{S}}.$

\noindent For the split quaternion $q={{q}_{0}}+{{q}_{1}}i+{{q}_{2}}j+{{q}_{3}}k\in {{\mathbb{H}}_{S}}$, the mapping:

\begin{equation}
L :{\mathbb{H}_S} \to {M_4}\left( \mathbb{R} \right),\,\,L \left( q \right) = \left( {\begin{array}{*{20}{c}}
{{q_0}}&{ - {q_1}}&\,\,\,\,{{q_2}}&\,\,\,\,{{q_3}}\\
{{q_1}}&\,\,\,\,{{q_0}}&\,\,\,\,{{q_3}}&{ - {q_2}}\\
{{q_2}}&\,\,\,\,{{q_3}}&\,\,\,\,{{q_0}}&{ - {q_1}}\\
{{q_3}}&{ - {q_2}}&\,\,\,\,{{q_1}}&\,\,\,\,{{q_0}}
\end{array}} \right)
\end{equation}

\noindent is an isomorphism amongst the algebra of matrices the above form and ${{\mathbb{H}}_{S}}$. $L \left( q \right)$ is called the left matrix representation for split quaternions $q\in {{\mathbb{H}}_{S}}.$

\noindent In a similar manner, we define the right matrix representation for the split quaternion $q,$ as  \cite{8}.

\begin{equation}
R :{\mathbb{H}_S} \to {M_4}\left( \mathbb{R} \right),\,\,R \left( q \right) = \left( {\begin{array}{*{20}{c}}
{{q_0}}&{ - {q_1}}&\,\,\,\,{{q_2}}&\,\,\,\,{{q_3}}\\
{{q_1}}&\,\,\,\,{{q_0}}&{ - {q_3}}&\,\,\,\,{{q_2}}\\
{{q_2}}&{ - {q_3}}&\,\,\,\,{{q_0}}&\,\,\,\,{{q_1}}\\
{{q_3}}&\,\,\,\,{{q_2}}&{ - {q_1}}&\,\,\,\,{{q_0}}
\end{array}} \right).
\end{equation}

\begin{theorem}
(\cite{8}) If $q,p\in {{\mathbb{H}}_{S}}$  and $c\in \mathbb{R}$, then we have:
\begin{description}
  \item[i.] 	$L \left( q+p \right)=L \left( q \right)+L \left( p \right),\,\,\,\,R \left( q+p \right)=R \left( q \right)+R \left( p \right).$
  \item[ii.] 	$L \left( cq \right)=cL \left( q \right),\,\,\,R \left( cq \right)=cR \left( q \right).$
  \item[iii.] $qp=L \left( q \right)p,\,\,\,\,\,qp=R \left( p \right)q,\,\,L \left( q \right)R \left( p \right)=R \left( p \right)L \left( q \right).$
  \item[iv.] $L \left( qp \right)=L \left( q \right)L \left( p \right),\,\,\,R \left( qp \right)=R \left( q \right)R \left( p \right).$
  \item[v.]	$L \left( {{q}^{-1}} \right)={{L }^{-1}}\left( q \right),\,\,\,\,R \left( {{q}^{-1}} \right)={{R }^{-1}}\left( q \right),\,\,\,\,\left\| q \right\|\ne 0.$
\end{description}
\end{theorem}

\begin{definition}
\noindent Two split quaternions $a$ and $b$  are said to be consimilar if there exists a split quaternion $p,\,\,\,\left\| p \right\|\ne 0$  so that $a=\overline{p}b{{p}^{-1}}\,;$ this is denoted as $a\overset{c}{\mathop{\sim }}\,b.$ Obviously, the consimilar split quaternions have the same norm. Also $\overset{c}{\mathop{\sim }}\,$ is an equivalence relation on the split quaternions.
\end{definition}

\begin{theorem}
(\cite{11})   Let $a={{a}_{0}}+{{a}_{1}}i+{{a}_{2}}j+{{a}_{3}}k$ and $b={{b}_{0}}+{{b}_{1}}i+{{b}_{2}}j+{{b}_{3}}k$ be two  split quaternions, $a$ and $b$ are spacelike (or timelike). The equation

\begin{equation} ax = \overline x b
\end{equation}
\noindent has a solution $x\in {{\mathbb{H}}_{S}}$ which  $\left\| x \right\| \ne 0$ the necessary and sufficient condition $\left\| a \right\|=\left\| b \right\|.$

\noindent Under this circumstances, if $a+\overline{b}\ne 0,$ the equation (10)  has a solution

\[x=\lambda \left( \overline{a}+b \right)\]

\noindent where $\lambda \in \mathbb{R}$. If $a+\overline{b}=0,$ the equation (10) has a solution

\[x={{x}_{0}}+{{x}_{1}}i+{{x}_{2}}j+{{x}_{3}}k\]

\noindent where ${{a}_{0}}{{x}_{0}}-{{a}_{1}}{{x}_{1}}+{{a}_{2}}{{x}_{2}}+{{a}_{3}}{{x}_{3}}=0.$

\end{theorem}

\begin{corollary}
\noindent From the Theorem 2.3, we can write $a=\overline{p}\left\| a \right\|{{p}^{-1}},$ for split quaternion $a$, such that $\left\| a \right\| \ne 0\,$ and $a \notin \mathbb{R},$ where $p=\left\| a \right\|+\overline{a}.$

\end{corollary}

\begin{theorem}
(\cite{11})  Let $a\in {{\mathbb{H}}_{S}}$ where $\left\| a \right\| \ne 0\,$ and $a \notin \mathbb{R}$. Then the quadratic equation ${{x}^{2}}=a$ has two split quaternion solutions as follow

\begin{equation}
x =  \pm \frac{{{{\left\| a \right\|}^{\frac{1}{2}}}\left( {\left\| a \right\| + a} \right)}}{{\left\| {\left\| a \right\| + a} \right\|}} =  \pm \left( {{\lambda _0} + {\lambda _1}a} \right),
\end{equation}

\noindent where ${{\lambda }_{0}}=\frac{{{\left\| a \right\|}^{\frac{3}{2}}}}{\left\| \left\| a \right\|+a \right\|},\,\,\,{{\lambda }_{1}}=\frac{{{\left\| a \right\|}^{\frac{1}{2}}}}{\left\| \left\| a \right\|+a \right\|}.$
\end{theorem}

\section{Consimilarity of Split Quaternion Matrices}

\noindent The set of $m\times n$ matrices with the split quaternion entries, which is denoted by $\mathbb{H}_{S}^{m\times n}$ with ordinary matrix addition and multiplication is a ring with unity. Let ${{A}^{T}}$, $\overline{A}$ and ${{A}^{*}}=\overline{\left( {{A}^{T}} \right)}$ be transpose, conjugate and   transpose conjugate matrix of $A\in \mathbb{H}_{S}^{n\times n}, respectively.$

\begin{theorem}
(\cite{12}) For any $A\in \mathbb{H}_{S}^{m\times n}$ and $B\in \mathbb{H}_{S}^{n\times s},$ the followings statements are valid:

\begin{description}
  \item[i.] 	${{\left( \overline{A} \right)}^{T}}=\overline{\left( {{A}^{T}} \right)}\,\,;$
  \item[ii.] 	${{\left( AB \right)}^{*}}={{B}^{*}}\,{{A}^{*}};$
  \item[iii.] If $A$ and $B$ are nonsingular,  ${{\left( AB \right)}^{-1}}={{B}^{-1}}{{A}^{-1}};$
  \item[iv.] If $A$ is nonsingular,  ${{\left( {{A}^{*}} \right)}^{-1}}={{\left( {{A}^{-1}} \right)}^{*}};$
  \item[v.] If $A$ is nonsingular, ${{\left( \overline{A} \right)}^{-1}}\ne \overline{\left( {{A}^{-1}} \right)}$, 	in general;
  \item[vi.] If $A$ is nonsingular, ${{\left( {{A}^{T}} \right)}^{-1}}\ne {{\left( {{A}^{-1}} \right)}^{T}}$, 	in general;
  \item[vii.] $\overline{AB}\ne \overline{A}\,\overline{B}$, 	in general;
  \item[viii.] ${{\left( AB \right)}^{T}}\ne {{B}^{T}}{{A}^{T}},$ 	in general.
\end{description}
\end{theorem}

\noindent A matrix $A\in\mathbb{ H}_{S}^{n\times n}$ is said to be similar a matrix $B\in \mathbb{H}_{S}^{n\times n}$ if there exists a nonsingular matrix $P\in \mathbb{H}_{S}^{n\times n}$ so that ${{P}^{-1}}AP=B.$ The relation, $A$ is similar to $B$, is denoted $A\sim B.$ Similarity is an equivalence relation on the split quaternion matrices.\\

\noindent On the other hand, a complex matrix $A\in {{\mathbb{C}}^{n\times n}}$ is said to be complex consimilar $B\in {{\mathbb{C}}^{n\times n}}$ if a nonsingular matrix $\overline{P}A{{P}^{-1}}=B$ be found. Complex consimilarity is an equivalence relation on ${{\mathbb{C}}^{n\times n}}$ and has been extensively studied \cite{13}. The split quaternion holds an important place in differential geometry and structure theory of Lorentz spaces \cite{8}-\cite{9}, and for this reason consimilarity for complex matrices will be extended for split quaternion matrices.

\noindent If $A,B\in \mathbb{H}_{S}^{n\times n},$  generally ${{\left( AB \right)}^{T}}\ne {{B}^{T}}{{A}^{T}},$ and $\overline{AB}\ne \overline{A}\,\overline{B}.$ Thus the mapping $A\to \overline{P}A{{P}^{-1}}$ is not an equivalence relation on $\mathbb{H}_{S}^{n\times n}.$ Thus we need to give a new definition of consimilarity of split quaternion matrices.

\begin{definition}
\noindent Let $A\in \mathbb{H}_{S}^{n\times n},$  then we define  $\widetilde{A}=jAj.$ We say that $\widetilde{A}$ is the $j-conjugate$ of $A.$
\end{definition}
\noindent For any $A,B\in \mathbb{H}_{S}^{m\times n}$ and $C\in \mathbb{H}_{S}^{n\times s},$ the following equalities are easy to confirm

\begin{description}
  \item[i.] 	$\widetilde{\left( \widetilde{A} \right)}=A;$
  \item[ii.] 	$\widetilde{\,\,\left( A+B \right)}=\widetilde{A}+\widetilde{B};$
  \item[iii.] 	$\widetilde{\left( AC \right)}=\widetilde{A}\widetilde{C};$
  \item[iv.] 	$\overline{\left( \widetilde{A} \right)}=\widetilde{\left( \overline{A} \right)}.$
\end{description}

\begin{theorem}
\noindent If $A\in \mathbb{H}_{S}^{n\times n},$ in that case
\[A\,\,is\,\,nonsingular\,\,\Leftrightarrow \,\,\widetilde{A}\,\,is\,\,nonsingular\,\,\Leftrightarrow {{A}^{*}}\,is\,\,nonsingular.\]

\noindent Furthermore, if $A$ is nonsingular,   ${\left( {{A^*}} \right)^{ - 1}} = {\left( {{A^{ - 1}}} \right)^*}$ and ${\left( {\widetilde A} \right)^{ - 1}} = \widetilde {\left( {{A^{ - 1}}} \right)}$.

\end{theorem}

\begin{proof}
\noindent Since $A$ is nonsingular, there exists an matrix ${{A}^{-1}}\in \mathbb{H}_{S}^{n\times n}$ so that $A{{A}^{-1}}=I$. Thus $A{{A}^{-1}}=I\Leftrightarrow jAjj{{A}^{-1}}j=I$ and ${{\left( \widetilde{A} \right)}^{-1}}=\widetilde{\left( {{A}^{-1}} \right)}.$ Similar way, we get ${{A}^{*}}{{\left( {{A}^{-1}} \right)}^{*}}=I.$
\end{proof}

\begin{definition}
\noindent Two split quaternion matrices $A$ and $B$ are expressed as consimilar if nonsingular split quaternion matrix $P$ so that $\widetilde{P}A{{P}^{-1}}=B$ be found, this is denoted as  $A\overset{c}{\mathop{\sim }}\,B.$
\end{definition}

\begin{theorem}
\noindent For $A,B,C\in \mathbb{H}_{S}^{n\times n},$ the followings are satisfied:
\begin{itemize}
  \item Reflexive:  $A\overset{c}{\mathop{\sim }}\,A\,;$
  \item Symmetric: if $A\overset{c}{\mathop{\sim }}\,B,$ then $B\overset{c}{\mathop{\sim }}\,A;$
  \item Transitive: if $A\overset{c}{\mathop{\sim }}\,B \, and \, B\overset{c}{\mathop{\sim }}\,C$ then $A\overset{c}{\mathop{\sim }}\,C.$
\end{itemize}
\end{theorem}

\begin{proof}
\noindent \\
\begin{itemize}
  \item Reflexive:  $\widetilde{{{I}_{n}}}\,A\,I_{n}^{-1}=A$ trivially, for $A\in \mathbb{H}_{S}^{n\times n}.$ So, consimilarity is reflexive.
  \item Symmetric: Let $\widetilde{P}A{{P}^{-1}}=B.$ As $P$ is nonsingular, we have \[\begin{array}{l}
{\left( {\widetilde P} \right)^{ - 1}}BP = {\left( {\widetilde P} \right)^{ - 1}}\widetilde PAP{P^{ - 1}}\\
\,\,\,\,\,\,\,\,\,\,\,\,\,\,\,\,\,\,\,\,\,\,\,\,\,\,\,\,\,= {I_n}A{I_n}\\
\,\,\,\,\,\,\,\,\,\,\,\,\,\,\,\,\,\,\,\,\, \,\,\,\,\,\,\,\,= A.
\end{array}\]

\noindent So, consimilarity is symmetric.
  \item 	Transsitive:  Let $\widetilde{{{P}_{1}}}A{{P}_{1}}^{-1}=B$ and $\widetilde{{{P}_{2}}}B{{P}_{2}}^{-1}=C.$ Then \[\begin{array}{l}
C = \widetilde {{P_2}}\widetilde {{P_1}}AP_1^{ - 1}P_2^{ - 1}\\
\,\,\,\,\,\, = \left( {\widetilde {{P_2}{P_1}}} \right)A{\left( {{P_2}{P_1}} \right)^{ - 1}}.
\end{array}\]
\noindent So, consimilarity is transitive.
\end{itemize}

\end{proof}

\noindent Then, by Theorem 3.3  consimilarity is an equivalence relation on $\mathbb{H}_{S}^{n\times n}.$ Clearly if $A\in {{\mathbb{C}}^{n\times n}}$, then $\overline{A}=\widetilde{A}=jAj.$ Thus, $A\in {{\mathbb{C}}^{n\times n}}$ is consimilar to $B\in {{\mathbb{C}}^{n\times n}}$ as complex matrices if $A$ is consimilar to $B$ as split quaternion matrices. Then, consimilarity relation in  $\mathbb{H}_{S}^{n\times n}$ is a natural extension of complex consimilarity in ${{\mathbb{C}}^{n\times n}}$.

\begin{theorem}
\noindent If $A,B\in \mathbb{H}_{S}^{n\times n},$ then \[A\overset{c}{\mathop{\sim }}\,B\,\,\Leftrightarrow \,\,jA\sim jB\,\,\Leftrightarrow Aj\sim Bj\,\,\Leftrightarrow jA\sim Bj.\]

\end{theorem}

\begin{proof}
\noindent Since $A\overset{c}{\mathop{\sim }}\,B\,\,\Leftrightarrow $ there exists a nonsingular matrix $P\in \mathbb{H}_{S}^{n\times n}$ so that $\widetilde{P}A{{P}^{-1}}=jPjA{{P}^{-1}}=B.$ Thus $A\overset{c}{\mathop{\sim }}\,B\,\,\Leftrightarrow PjA{{P}^{-1}}=jB\Leftrightarrow jA\sim jB.$
Since ${{j}^{-1}}jAj=Aj,\,$ we get $\,jA\sim Aj$ and $jB\sim Bj.$ Therefore $jA\sim jB\Leftrightarrow Aj\sim Bj\Leftrightarrow jA\sim Bj.$

\end{proof}

\begin{definition}
\noindent Let $A\in \mathbb{H}_{S}^{n\times n},\,\lambda \in {{\mathbb{H}}_{S}}.$ If there exists $0\ne x\in \mathbb{H}_{S}^{n\times 1}$ such that
\[A\widetilde x = x\lambda \,\,\left( {A\widetilde x = \lambda x} \right)\]

\noindent then $\lambda $ is said to be a right (left) coneigenvalues of $A$ and $x$ is said to be a coneigenvector of $A$ corresponding to the right (left) coneigenvalue $\lambda .$ The set of right coneigenvalues is defined as

\[\widetilde {{\sigma _r}}\left( A \right)  = \left\{ {\lambda  \in {\mathbb{H}_S}:\,A\widetilde x = x\lambda ,\,for\,\,some\,\,x \ne 0} \right\}.\]

\noindent The set of left coneigenvalues is similarly defined and is denoted by $\widetilde {{\sigma _l}}\left( A \right)$ .

\end{definition}

\noindent Recall that if $ x\in \mathbb{H}_{S}^{n\times 1} (x \ne 0),$ and $\lambda  \in {\mathbb{H}_S}$ satisfying $Ax=x\lambda \,\,\left( Ax=\lambda x \right)$, we call $x$ an eigenvector of $A$, while $\lambda$ is an right (left) eigenvalue of $A.$ We also say that $x$ is an eigenvector corresponding to the right(left) eigenvalue $\lambda.$

\begin{theorem}
Let $A,B\in \mathbb{H}_{S}^{n\times n},$ if  $A$ is consimilar to $B$, then $A$ and $B$ have the same right coneigenvalues.
\end{theorem}

\begin{proof}
\noindent Let $A\overset{c}{\mathop{\sim }}\,B,$ then, there exists a nonsingular matrix $P\in \mathbb{H}_{S}^{n\times n}$ such that $B=\widetilde{P}A{{P}^{-1}}.$ Let $\lambda \in {{\mathbb{H}}_{S}}$ be a right coneigenvalue for the matrix $A,$ then we find the matrix $\,0 \ne \, x\,\in \mathbb{H}_{s}^{n\times 1}\,$ such that $\,A\widetilde x \,=\, x\lambda.\,$ Let $y=P\widetilde{x}.\,$ Finally $\,By=\widetilde{P}A{{P}^{-1}}y=\widetilde{P}A\widetilde{x}=\widetilde{P}x\lambda =\widetilde{y}\lambda .$
\end{proof}

\begin{theorem}
\noindent If $A\in \mathbb{H}_{S}^{n\times n},$ in that case $\lambda $ is right  coneigenvalue of $A$ necessary and sufficient condition for any $\beta  \in {\mathbb{H}_S}\,\,\left( {\left\| \beta  \right\| \ne 0} \right),\,\,\,\,\,\widetilde{\beta }\lambda {{\beta }^{-1}}$ is a right  coneigenvalue of $A.$
\end{theorem}

\begin{proof}
\noindent From $A\widetilde{x}=x\lambda \,,$ we get $A\left( \widetilde{x}\beta  \right)=x\,\widetilde{\beta }\,\,{{\left( \widetilde{\beta } \right)}^{-1}}\lambda \beta .$
\end{proof}

\begin{theorem}
\noindent If $A\in \mathbb{H}_{S}^{n\times n}\,\,\text{and}\,\,\,\,\lambda \in {{\mathbb{H}}_{S}},$ then
${{\lambda }_{0}}$ is a right coneigenvalue of $A\,\,\Leftrightarrow \,\,j{{\lambda }_{0}}$ is a right eigenvalue of $A\,j\,\Leftrightarrow \,\,{{\lambda }_{0}}j$ is a right eigenvalue of $jA.$

\end{theorem}

\begin{proof}
\noindent Suppose that  ${{\lambda }_{0}}$ is right coneigenvalue of $A$. Then $0\ne x\in \mathbb{H}_{S}^{n\times n}$ so that $A\widetilde{x}=Ajxj=x{{\lambda }_{0}}\,\,\Leftrightarrow \,\,$ $Ajx=x\left( {{\lambda }_{0}}j \right)\,\Leftrightarrow {{\lambda }_{0}}j$ is a right eigenvalue of $Aj.$
Also,$A\widetilde{x}=x{{\lambda }_{0}}\,\,\,\Leftrightarrow \,\,jA\widetilde{x}=jxjj{{\lambda }_{0}}=\widetilde{x}j{{\lambda }_{0}}\,\Leftrightarrow \,\,j{{\lambda }_{0}}$ is a right eigenvalue of $jA. $

\end{proof}

\begin{definition}
\noindent (\cite{12}) Let $A={{A}_{1}}+{{A}_{2}}j\in \mathbb{H}_{S}^{n\times n}$ where ${{A}_{s} \in}$ ${\mathbb{C}^{n{\rm{x}}n}}$, $s=1,2$. The $2n\times 2n$ matrix

\[\left( {\begin{array}{*{20}{c}}
{{A_1}}&\,\,\,\,{{A_2}}\\
{\overline {{A_2}} }&\,\,\,\,{\overline {{A_1}} }
\end{array}} \right)\]

\noindent is called the complex adjoint matrix of $A$ and denoted ${{\chi }_{A}}.$
\end{definition}

\noindent It is nearby to identify a split quaternion matrix $A\in \mathbb{H}_{S}^{n\times n}$ with a complex matrix $\textbf{A}\in {{\mathbb{C}}^{2n\times n}}.$ By the $\cong $ symbol, we will define

\[A = {A_1} + {A_2}j \cong \textbf{A} = \left( {\begin{array}{*{20}{c}}
{{A_1}}\\
{{A_2}}
\end{array}} \right) \in {\mathbb{C}^{^{2n \times n}}}.\]

\noindent Then, multiplication of $A\in \mathbb{H}_{S}^{n\times n}$ and $B\in \mathbb{H}_{S}^{n\times n}$ can be shown with the help of ordinary matrix multiplication $A\,B\cong {{\left( {{\chi }_{B}} \right)}^{T}}\,\textbf{A}.$

\begin{theorem}
\noindent (\cite{12}) Let $A,B\in \mathbb{H}_{S}^{n\times n},$ then the  followings are satisfied:

\begin{description}
  \item[i.]	${{\chi }_{A+B}}={{\chi }_{A}}+{{\chi }_{B}};$
  \item[ii.] 	${{\chi }_{AB}}={{\chi }_{A}}{{\chi }_{B}};$
  \item[iii.] 	If   is  nonsingular, ${{\left( {{\chi }_{A}} \right)}^{-1}}={{\chi }_{{{A}^{-1}}}};$
  \item[iv.] 	In general ${{\chi }_{{{A}^{*}}}}\ne {{\left( {{\chi }_{A}} \right)}^{*}}$.
\end{description}

\end{theorem}

\begin{theorem}
\noindent For every $A\in \mathbb{H}_{S}^{n\times n},$
\[{\widetilde \sigma _r}\left( A \right) \cap \mathbb{C }= \widetilde \sigma \left( {{\chi _A}} \right)\]
where $\widetilde \sigma \left( {{\chi _A}} \right) = \left\{ {\lambda  \in \mathbb{C}:\,\,{\chi _A}\overline y  = \lambda y,\,\,for\,\,some\,\,y \ne 0} \right\},$ is the set of coneigenvalues of ${{\chi }_{A}}.$
\end{theorem}

\begin{proof}
\noindent Let $A={{A}_{1}}+{{A}_{2}}j\in \mathbb{H}_{S}^{n\times n}$ such that ${{A}_{s} \in}$ ${\mathbb{C}^{n{\rm{x}}n}}$, $s=1,2$,  and $\lambda \in \mathbb{C}$ be a right coneigenvalue of $A.$ Therefore there exists nonzero column vector $x\in \mathbb{H}_{S}^{n\times 1}$ such that $A\widetilde{x}=x\lambda .$ This implies

\[\begin{array}{l}
\left( {{A_1} + {A_2}j} \right)\left( {\overline {{x_1}}  + \overline {{x_2}} j} \right) = \left( {{x_1} + {x_2}j} \right)\lambda \\
\left( {A\overline {{x_1}}  + {A_2}{x_2}} \right) = {x_1}\lambda \,\,\,\,\,\,\text{and}\,\,\,\,\,\left( {{A_2}{x_1} + {A_1}\overline {{x_2}} } \right) = {x_2}\overline \lambda
\end{array}\]
Using these equations, we can write
\[\left( {\begin{array}{*{20}{c}}
{{A_1}}&\,\,\,\,{{A_2}}\\
{\overline {{A_2}} }&\,\,\,\,{\overline {{A_1}} }
\end{array}} \right)\left( {\begin{array}{*{20}{c}}
{\overline {{x_1}} }\\
{{x_2}}
\end{array}} \right) = \lambda \left( {\begin{array}{*{20}{c}}
{{x_1}}\\
{\overline {{x_2}} }
\end{array}} \right).\]
\noindent Therefore, the complex right coneigenvalue of the split quaternion matrix $A$ is an equivalent to the coneigenvalue of the adjoint matrix ${{\chi }_{A}}$ that is $\widetilde {{\sigma _r}}\left( A \right) \cap \mathbb{C} = \widetilde \sigma \left( {{\chi _A}} \right).$

\end{proof}

\section{Real Representation of Split Quaternion Matrices}

\noindent Let $A={{A}_{0}}+{{A}_{1}}i+{{A}_{2}}j+{{A}_{3}}k\in \mathbb{H}_{S}^{m\times n}$ where ${{A}_{s}}$ are $m\times n$ real matrices, $s=0,1,2,3$. We define the linear transformation ${{\phi }_{A}}\left( X \right)=A\widetilde{X}.$ Then, we can write

\[\begin{array}{l}
{\phi _A}\left( 1 \right) = A = {A_0} + {A_1}i + {A_2}j + {A_3}k\\
{\phi _A}\left( i \right) = A\,\widetilde i = {A_1} - {A_0}i - {A_3}j + {A_2}k\\
{\phi _A}\left( j \right) = A\,\widetilde j = {A_2} + {A_3}i + {A_0}j + {A_1}k\\
{\phi _A}\left( k \right) = A\,\widetilde k =  - {A_3} + {A_2}i + {A_1}j - {A_0}k.
\end{array}\]

\noindent Then, we find the real representation of the split quaternion matrix $A$ as follows:

\[{\phi _A} = \left( {\begin{array}{*{20}{c}}
{{A_0}}&\,\,\,\,{{A_1}}&\,\,\,\,{{A_2}}&{ - {A_3}}\\
{{A_1}}&{ - {A_0}}&\,\,\,\,{{A_3}}&\,\,\,\,{{A_2}}\\
{{A_2}}&{ - {A_3}}&\,\,\,\,{{A_0}}&\,\,\,\,{{A_1}}\\
{{A_3}}&\,\,\,\,{{A_2}}&\,\,\,\,{{A_1}}&{ - {A_0}}
\end{array}} \right) \in {\mathbb{R}^{4m \times 4n}}.\]

\noindent It is nearby to identify a split quaternion matrix $A\in \mathbb{H}_{S}^{m\times n}$ with a real matrix $\textbf{A}\in {{\mathbb{R}}^{4m\times n}}.$
By the $\cong $ symbol, we will define

\[A = {A_0} + {A_1}i + {A_2}j + {A_3}k \cong \textbf{A} = \left( {\begin{array}{*{20}{c}}
{{A_0}}\\
{{A_1}}\\
{{A_2}}\\
{{A_3}}
\end{array}} \right) \in {\mathbb{R}^{^{4m \times n}}}.\]

\noindent Then, the multiplication of $A\in \mathbb{H}_{S}^{m\times n}$ and $\widetilde{B}\in {{\mathbb{H}}^{n\times k}}$ can be shown with the help of ordinary matrix multiplication $A\,\widetilde{B}\cong {{\phi }_{A}}\,\textbf{B}.$

\begin{theorem}
\noindent For the split quaternion matrix $A$, the following identities are satisfied:

\begin{description}
  \item[i.] If $A\in \mathbb{H}_{S}^{m\times n},$ then \[P_m^{ - 1}{\phi _A}{P_n} = {\phi _{\widetilde A}}\,,\,\,\,Q_m^{ - 1}{\phi _A}{Q_n} =  - {\phi _A},\,\,\,\,\,R_m^{ - 1}{\phi _A}{R_n} = {\phi _A},\,\,\,\,\,\,S_m^{ - 1}{\phi _A}{S_n} =  - {\phi _A};\]

\noindent where
  \[\begin{array}{l}
{P_m} = \left( {\begin{array}{*{20}{c}}
{{I_m}}&\,\,\,\,0&\,\,\,\,0&\,\,\,\,0\\
0&{ - {I_m}}&\,\,\,\,0&\,\,\,\,0\\
0&\,\,\,\,0&\,\,\,\,{{I_m}}&\,\,\,\,0\\
0&\,\,\,\,0&\,\,\,\,0&{ - {I_m}}
\end{array}} \right),\,\,\,\,\,\,\,\,\,\,\,\,\,\,\,\,\,\,\,\,\,\,\,\,\,\,\,\,\,{Q_m} = \left( {\begin{array}{*{20}{c}}
0&{ - {I_m}}&\,\,\,\,0&\,\,\,\,0\\
{{I_m}}&\,\,\,\,0&\,\,\,\,0&\,\,\,\,0\\
0&\,\,\,\,0&\,\,\,\,0&\,\,\,\,{{I_m}}\\
0&\,\,\,\,0&{ - {I_m}}&\,\,\,\,0
\end{array}} \right),\\
\,\\
{R_m} = \left( {\begin{array}{*{20}{c}}
\,\,\,\,0&\,\,\,\,0&{ - {I_m}}&\,\,\,\,0\\
\,\,\,\,0&\,\,\,\,0&\,\,\,\,0&{ - {I_m}}\\
{ - {I_m}}&\,\,\,\,0&\,\,\,\,0&\,\,\,\,0\\
\,\,\,\,0&{ - {I_m}}&\,\,\,\,0&\,\,\,\,0
\end{array}} \right),\,\,\,\,\,\,\,\,\,\,\,\,\,\,\,\,\,\,\,\,\,{S_m} = \left( {\begin{array}{*{20}{c}}
\,\,\,\,0&\,\,\,\,0&\,\,\,\,0&\,\,\,\,{{I_m}}\\
\,\,\,\,0&\,\,\,\,0&{ - {I_m}}&\,\,\,\,0\\
\,\,\,\,0&{ - {I_m}}&\,\,\,\,0&\,\,\,\,0\\
\,\,\,\,{{I_m}}&\,\,\,\,0&\,\,\,\,0&\,\,\,\,0
\end{array}} \right),
\end{array}\]

  \item[ii.] If $A,B\in \mathbb{H}_{S}^{m\times n},$ then  ${{\phi }_{A+B}}={{\phi }_{A}}+{{\phi }_{B}};$
  \item[iii.] If $A\in \mathbb{H}_{S}^{m\times n},\,\,B\in \mathbb{H}_{S}^{n\times r},$ in that case ${{\phi }_{AB}}={{\phi }_{A}}{{P}_{n}}{{\phi }_{B}}={{\phi }_{A}}{{\phi }_{\widetilde{B}}}{{P}_{r}};$
  \item[iv.] If $A\in \mathbb{H}_{S}^{m\times m},\,\,$ in that case $A$ is nonsingular necessary and sufficient condition ${{\phi }_{A}}$ is nonsingular, $\phi _{A}^{-1}=P_{m}{{\phi }_{{{A}^{-1}}}}P_{m}$;
  \item[v.] If  $A\in \mathbb{H}_{S}^{m\times m},\,$ ${{\phi }_{\overline{A}}}={{\varepsilon }_{2}}{{\left( {{\phi }_{A}} \right)}^{T}}{{\varepsilon }_{2}}$ where ${\varepsilon _2} = \left( {\begin{array}{*{20}{c}}
{{I_m}}&0&0&0\\
0&{ - {I_m}}&0&0\\
0&0&{ - {I_m}}&0\\
0&0&0&{{I_m}}
\end{array}} \right);$
  \item[vi.] If $A\in \mathbb{H}_{S}^{m\times m},$ \[{\widetilde \sigma _r}\left( A \right) \cap \mathbb{C} = \sigma \left( {{\phi _A}} \right)\]
where $\sigma \left( {{\phi _A}} \right) = \left\{ {\lambda  \in \mathbb{C}:\,\,{\phi _A}y = \lambda y,\,\,for\,\,some\,\,y \ne 0} \right\}$  is the set of eigenvalues of ${{\phi }_{A}}.$
\end{description}

\end{theorem}

\begin{proof}
\noindent The first five statements can be seen in an easy way. Thus, we will prove  vi.

\noindent Let $A={{A}_{0}}+{{A}_{1}}i+{{A}_{2}}j+{{A}_{3}}k\in \mathbb{H}_{S}^{m\times m}$   where ${{A}_{s} \in}$ ${\mathbb{R}^{m{\rm{x}}m}}$ for $s=0,1,2,3$, and $\lambda \in \mathbb{C}$ be a right coneigenvalue of $A.$ Then, there exists nonzero column vector $x\in \mathbb{H}_{S}^{m\times 1}$ such that $A\widetilde{x}=x\lambda .$ We can write ${{\phi }_{A}}\,\textbf{x}=\textbf{x}\lambda .$ Then complex right coneigenvalue of split quaternion matrix $A$ is an equivalent to the eigenvalue of ${{\phi }_{A}}$ that is ${\widetilde \sigma _r}\left( A \right) \cap \mathbb{C} = \sigma \left( {{\phi _A}} \right).$

\end{proof}

\section{The Split Quaternion Matrix Equation $X-A\widetilde{X}B=C$}

\noindent In this part, we take into consideration the split quaternion matrix equation

\begin{equation}
X-A\widetilde{X}B=C
\end{equation}

\noindent through the real representation, where $A\in \mathbb{H}_{S}^{m\times m},\,\,B\in \mathbb{H}_{S}^{n\times n}$ and $C\in \mathbb{H}_{S}^{m\times n}.$

\noindent We define the real representation matrix equation of the split quaternion matrix equation (12) by

\begin{equation}
Y - {\phi _A}Y{\phi _B} = {\phi _C}
\end{equation}

\noindent By (iii.) in Theorem 4.1, the equation (12) is equivalent to the equation
\begin{equation}
{{\phi }_{X}}-{{\phi }_{A}}{{\phi }_{X}}{{\phi }_{B}}={{\phi }_{C}}.
\end{equation}

\begin{theorem}
\noindent The split quaternion matrix equation (12) has a solution $X$ necessary and sufficient condition real matrix equation (13) has a solution $Y={{\phi }_{X}}.$
\end{theorem}

\begin{theorem}
\noindent Let $A\in \mathbb{H}_{S}^{m\times m},\,\,B\in \mathbb{H}_{S}^{n\times n}$ and $C\in \mathbb{H}_{S}^{m\times n}.$ The split quaternion matrix equation (12) has a solution $X\in \mathbb{H}_{S}^{m\times n}$ necessary and sufficient condition the matrix equation (13) has a solution  $Y\in {{\mathbb{R}}^{4m\times 4n}}$. In this case, if $Y$ is a solution to the matrix equation (13), we have
\begin{equation}
 X = \frac{1}{{16}}\left( {{I_m}\,\,\,\,i{I_m}\,\,\,j{I_m}\,\,\,k{I_m}} \right)\left( {Y - Q_m^{ - 1}Y{Q_n} + R_m^{ - 1}Y{R_n} - S_m^{ - 1}Y{S_n}} \right)\left( {\begin{array}{*{20}{c}}
{{I_m}}\\
{i{I_m}}\\
{j{I_m}}\\
{ - k{I_m}}
\end{array}} \right)
\end{equation}

\noindent is a solution to (12).

\end{theorem}
\begin{proof}
\noindent We show that if the real matrix

\[Y = \left( {\begin{array}{*{20}{c}}
{{Y_{11}}}&\,\,\,{{Y_{12}}}&\,\,\,{{Y_{13}}}&\,\,\,{{Y_{14}}}\\
{{Y_{21}}}&\,\,\,{{Y_{22}}}&\,\,\,{{Y_{23}}}&\,\,\,{{Y_{24}}}\\
{{Y_{31}}}&\,\,\,{{Y_{32}}}&\,\,\,{{Y_{33}}}&\,\,\,{{Y_{34}}}\\
{{Y_{41}}}&\,\,\,{{Y_{42}}}&\,\,\,{{Y_{43}}}&\,\,\,{{Y_{44}}}
\end{array}} \right),\,\,{Y_{uv}} \in {\mathbb{R}^{m \times n}},\,u,v = 1,2,3,4\]

\noindent is a solution to (13), the matrix represented in (15) is a solution to (12). Since\\  $\,Q_{m}^{-1}Y{{Q}_{n}}=-Y,\,\,R_{m}^{-1}Y{{R}_{n}}=Y,$ and $S_{m}^{-1}Y{{S}_{n}}=-Y,$ we have

\begin{equation}
\begin{array}{l}
- Q_m^{ - 1}Y{Q_n} - {\phi _A}\left( { - Q_m^{ - 1}Y{Q_n}} \right){\phi _B} = {\phi _C}\\
\,\,\,\,\,R_m^{ - 1}Y{R_n}\, - {\phi _A}\left(\, {\,R_m^{ - 1}Y{R_n}\,} \right){\phi _B}\, = {\phi _C}\\
 - S_m^{ - 1}Y{S_n}\,\, - \,{\phi _A}\left( { - S_m^{ - 1}Y{S_n}} \right){\phi _B} = {\phi _C}.
\end{array}
\end{equation}

\noindent Last equations  show that if $Y$ is a solution to (13), then $\,-Q_{m}^{-1}Y{{Q}_{n}},\,\,\,\,\,R_{m}^{-1}Y{{R}_{n}}$ and $S_{m}^{-1}Y{{S}_{n}}$ are also solutions to (13). Then the undermentioned real matrix:

\begin{equation}
Y' = \frac{1}{4}\left( {Y - \,Q_m^{ - 1}Y\,{Q_n} + R_m^{ - 1}Y\,{R_n} - \,S_m^{ - 1}Y\,{S_n}} \right)
\end{equation}

\noindent is a solution to (13). After calculating, we easily obtain

\[Y' = \left( {\begin{array}{*{20}{c}}
{{Y'_0}}&\,\,\,\,{{Y'_1}}&\,\,\,\,{{Y'_2}}&{ - {Y'_3}}\\
{{Y'_1}}&{ - {Y'_0}}&\,\,\,\,{{Y'_3}}&\,\,\,\,{{Y'_2}}\\
{{Y'_2}}&{ - {Y'_3}}&\,\,\,\,{{Y'_0}}&\,\,\,\,{{Y'_1}}\\
{{Y'_3}}&\,\,\,\,{{Y'_2}}&\,\,\,\,{{Y'_1}}&{ - {Y'_0}}
\end{array}} \right),\]

\noindent where

\begin{equation}
\begin{array}{l}
{Y'_0} = \frac{1}{4}\left( {{Y_{11}} - {Y_{22}} + {Y_{33}} - {Y_{44}}} \right),\,\,\,\,{Y'_1} = \frac{1}{4}\left( {{Y_{12}} + {Y_{21}} + {Y_{34}} + {Y_{43}}} \right),\,\,\,\\
\,\\
{Y'_2} = \frac{1}{4}\left( {{Y_{13}} + {Y_{24}} + {Y_{31}} + {Y_{42}}} \right),\,\,\,\,{Y'_3} = \frac{1}{4}\left( { - {Y_{14}} + {Y_{23}} - {Y_{32}} + {Y_{41}}} \right).
\end{array}
\end{equation}

\noindent From (18), we formulate a split quaternion matrix:

\[X = {Y'_0} + {Y'_1}i + {Y'_2}j + {Y'_3}k = \frac{1}{4}\left( {{I_m}\,\,\,\,i{I_m}\,\,\,j{I_m}\,\,\,k{I_m}} \right)Y'\left( {\begin{array}{*{20}{c}}
{{I_n}}\\
{i{I_n}}\\
{j{I_n}}\\
{ - k{I_n}}
\end{array}} \right).\]

\noindent Clearly ${{\phi }_{X}}=Y'.$ By Theorem 5.1, $X$ is a solution to equation given by (12).
\end{proof}

\end{document}